\documentclass[12pt]{amsart}
\usepackage{amsfonts}
\usepackage{amsmath}
\usepackage{amssymb}
\usepackage[margin=1.2in]{geometry}
\usepackage{tikz}
\usepackage{mathrsfs,euscript,anyfontsize}
\usepackage{mathtools}
\usepackage{commath}
\usepackage{enumitem}

\newcommand{\eT}{{\EuScript T}}

\newcommand{\R}{\mathbb{R}}
\newcommand{\C}{\mathbb{C}}
\newcommand{\N}{\mathbb{N}}
\newcommand{\inv}{^{-1}}
\renewcommand{\Re}{\operatorname{Re}}
\renewcommand{\Im}{\operatorname{Im}}

\newtheorem{theorem}{Theorem}[section]

\newtheorem{lemma}[theorem]{Lemma}

\theoremstyle{definition}

\newtheorem{remark}[theorem]{Remark}
\newtheorem{example}[theorem]{Example}

\makeatletter
\@namedef{subjclassname@2020}{
  \textup{2020} Mathematics Subject Classification}
\makeatother

\numberwithin{equation}{section}

\begin{document}
\title[A general quantified Wiener-Ikehara theorem]{A general quantified Wiener-Ikehara theorem}

\author[G. Debruyne]{Gregory Debruyne}
\thanks{G. Debruyne gratefully acknowledges support by a senior postdoctoral fellowship of Research--Foundation Flanders through the grant number 1249624N}
\address{G. Debruyne\\ Department of Mathematics: Analysis, Logic and Discrete Mathematics\\ Ghent University\\ Krijgslaan 281 Gebouw S8\\ B 9000 Gent\\ Belgium}
\email{gregory.debruyne@UGent.be}

\begin{abstract} Building on the technology developed in \cite{Deb-qik}, we establish a quantified Tauberian theorem under a flexible Wiener-Ikehara Tauberian condition and under several types of boundary behavior for the Laplace transform. \end{abstract} 

\maketitle

\section{Introduction}

The Wiener-Ikehara theorem is one of the most celebrated results in Tauberian theory and in general, classical analysis. One of the most frequently cited versions reads as follows.

\begin{theorem} \label{qwi: thold}Let $\tau: [0,\infty) \rightarrow \R$ be a non-decreasing function for which the Laplace transform
\begin{equation}
 \mathcal{L}\{\tau;s\} = \int^{\infty}_{0} e^{-sx} \tau(x) \dif x \quad \text{converges for } \Re s > 1
\end{equation}
and for which $\mathcal{L}\{\tau;s\} - A/(s-1)$ admits a continuous extension to the line $\Re s = 1$ for a certain $A > 0$. Then,
\begin{equation}
 \tau(x) \sim A e^{x}, \ \ \ x \rightarrow \infty.
\end{equation}
\end{theorem}

One of the most well-known applications of this theorem is the deduction of the prime number theorem in the form $\pi(x) \sim x/\log x$, from the absence of $\zeta$-zeros on the line $\Re s = 1$. More generally, the Wiener-Ikehara theorem can often be applied in number theory to derive asymptotic formulas for sum functions of non-negative arithmetic functions. The Wiener-Ikehara theorem offers a powerful tool, especially in the subfield of Beurling prime number theory \cite{bateman-diamond, Debruyne-VindasPNTEquivalences, d-v-l1, s-v, tranoy-vindas, zhang2014} where the Laplace transform of the relevant functions may not admit an analytic extension and thus alternative approaches based on switching contours in the Perron inversion formula become inadequate.

Its realm of applications, however, extends beyond number theory. For example, there are applications in spectral asymptotics \cite{aramaki1988, shubin2001}, renewal theory \cite{chen-vindas, koga} and stability questions in dynamical systems \cite{petkov-stoyanov}, among many more. We refer to \cite[Ch. III]{korevaarbook} and \cite{p-t-z} for more extensive discussions of the literature of the Wiener-Ikehara theorem and some related Tauberian theorems. 

Obviously quantified versions of the Wiener-Ikehara theorem are then valuable in those applications if one desires a better error term. The literature for quantified versions of the Wiener-Ikehara theorem can however be hard to track down \cite{hartlappthesis}. The relevant papers often focus on a particular application \cite{aramaki1988, d-v, malliavin, nyman, petkov-stoyanov} and the Tauberian theorems are then designed specifically for that application, with no intention to provide the most general version that may apply to other contexts. Researchers may also find it difficult to recognize exactly which theorem applies to their setting \cite{p-t-z}. We therefore think it is worthwhile to record here a very general but still fairly accessible version of a quantified Wiener-Ikehara theorem. 

We briefly review some types of boundary behavior for the Laplace transform that have been considered previously in the literature. In \cite{aramaki1988, petkov-stoyanov}, one studies analytic extensions to a half-plane of the Laplace transform together with polynomial bounds. Hypotheses involving more general behavior on the analytic extensions are discussed in \cite{hartlappthesis}. In \cite{d-v, nyman} one studies certain $C^{\infty}$-boundary behavior while in \cite{malliavin} one exploits ultradifferentiable boundary behavior of the Laplace transform.

In \cite[Ch. II.7]{tenenbaumbook}, one does provide an effective/quantified version of the Wiener-Ikehara theorem, which was later generalized in \cite{revesz-roton}, but the assumptions on the Laplace transform are quite weak, roughly corresponding to at best $C^1$-boundary behavior, and it appears that at best one can only achieve the error term $O(e^x/x)$ for $\tau$ if this behavior were considered in Theorem \ref{qwi: thold}. Our main theorem in Section \ref{qwi sec:mainth} covers all aforementioned types of boundary behavior.

Closely related to the Wiener-Ikehara theorem is the Ingham-Karamata theorem. Loosely speaking, in an Ingham-Karamata theorem the abscissa of (absolute) convergence of the Laplace transform of a non-decreasing function is $0$, while for a Wiener-Ikehara theorem it is strictly larger than $0$. Recently, the author has developed an extensive remainder theory for the Ingham-Karamata theorem under a flexible one-sided Tauberian condition. The main goal of this paper is to present the general quantified Wiener-Ikehara theorem that can be established with the newly developed technology from \cite{Deb-qik}.

We should mention that sometimes one may perform some manipulations to reduce a Wiener-Ikehara argument to an Ingham-Karamata one. For example, one may consider the function $S(x) := \int^{x}_{0} e^{-u} \dif \tau(u)$ with $\mathcal{L}\{S;s\} = (1 + s^{-1}) \mathcal{L}\{\tau; s+ 1\}$. If $\tau$ is non-decreasing, then so is $S$ and asymptotics for $\tau$ may be recovered from those for $S$ through integration by parts. Another possible manipulation is to consider $S(x) = e^{-x} \tau(x)$ as done in \cite[Th. 1.4]{Deb-qik}, but this has as down-side that $S$ is not necessarily non-decreasing. However, these manipulations may cause some loss of information; the factor $s^{-1}$ in the first manipulation may for instance worsen the available regularity of the Laplace transform at $s = 0$, while also the integration by parts argument may result in a worse error term. We therefore still think there is intrinsic interest to provide a direct argument for the quantified Wiener-Ikehara theorem. Furthermore, we anticipate that researchers may also more easily recognize potential applications if the Tauberian theorem is written down into a proper Wiener-Ikehara setting. 

In \cite[Th. 1.4]{Deb-qik} we already communicated a model quantified Wiener-Ikehara theorem. Therein, merely to illustrate the strength of the innovative techniques developed in the paper, we discuss a version in which the analytic extension of the Laplace transform has only a simple pole at $s = 1$. It is mentioned that it should be possible to handle a more flexible Wiener-Ikehara Tauberian condition, allowing one to treat more versatile singularities (and also other boundary behavior), but this requires a more delicate argument and was postponed. This is the subject of this paper. We also mention \cite{delange1954} where a non-quantified version of the Wiener-Ikehara theorem is given able to consider other singularities than a simple pole.

We will use a pure Fourier method. The advantages of this technique have been extensively discussed in \cite{Deb-qik}. We only mention here that this approach splits the treatment of the Tauberian condition and that of the boundary behavior of the Laplace transform. As such, the analysis of the boundary behavior is the same as in \cite{Deb-qik} and is captured below by Lemma \ref{qwi: lemtransformcond}. In fact, we will focus here primarily on the Tauberian condition, which is more cumbersome to deal with than in \cite{Deb-qik}. We also remark briefly that the sequence of test functions $\phi_n$ that was one of the main technical innovations in \cite{Deb-qik} is also present in this work, but is tacitly contained in Lemma \ref{qwi: lemtransformcond}.

The Tauberian condition we shall work with is modelled after $e^{x}(S(x) + F(x))$ being non-decreasing. In comparison with Theorem \ref{qwi: thold} we impose $S(x) = e^{-x} \tau(x) -A$ to translate the Laplace boundary behavior to $\Re s = 0$ and to subtract the singularity at $s= 1$. The function $F$, as discussed in \cite{Deb-qik}, incorporates the desired flexibility necessary for varying applications, and may be used to encode the singularities of the Laplace transform; in the setting of Theorem \ref{qwi: thold} one chooses $F(x) = A$. We shall actually treat a more general Tauberian condition, corresponding to $h(x) (S(x) + F(x))$ being non-decreasing, where $h(x)$ does not necessarily have to be $e^x$. The reason for incorporating this additional flexibility in the Tauberian condition is to cover some twisted Laplace transforms such as $\int^{\infty}_{0} x^{-\beta} e^{-sx} \tau(x) \dif x$ for a non-decreasing function $\tau$. These for example occur in \cite{katsuda-tsunada} and in the proof of \cite{petkov-stoyanov}.

As in \cite{Deb-qik}, one may anticipate that the error terms obtained in our main Theorem \ref{qwi: thmain} are in most cases best possible with respect to the considered assumptions, but we have not pursued to capture this is in a rigorous statement as in \cite[Th. 5.1]{Deb-qik}. 

As a small sample result, we record here a quantified Wiener-Ikehara theorem involving an analytic continuation with one pole-like singularity.

\begin{theorem} Let $\tau: [0,\infty) \rightarrow \mathbb{R}$ be a non-decreasing function, $\beta, \rho, A >0$ and $G,H: [0,\infty) \rightarrow \mathbb{R}$ two continuous functions such that $G(t)$ and $tH(t)$ are non-decreasing. Suppose that $\mathcal{L}\{\tau;s\} = \int^{\infty}_0 e^{-sx} \tau(x) \dif x$ converges on $\Re s > \beta$ and that
\begin{equation}
 \mathcal{L}\{\tau; s \} - \frac{A}{(s-\beta)^{\rho}}
\end{equation}
admits an analytic continuation to the region $\{s : |\Re s - \beta| \leq 1/G(|\Im s|) \}$ where it satisfies the bound $H(|s - \beta|)$.
Then, as $x\rightarrow \infty$,
\begin{equation} \label{qwieq: sam}
 \tau(x) = \frac{Ae^{\beta x} x^{\rho - 1}}{\Gamma(\rho)} + O\bigg(e^{\beta x} \inf_{\lambda \geq 1} \left\{\frac{x^{\rho - 1}}{\lambda} + \exp\bigg(-\frac{x}{G(\lambda) + 675 \lambda^{\inv}}\bigg) \int^{\lambda}_{0} H(t) \dif t \right\}  \bigg).\end{equation}
\end{theorem}

For example, if $G(x) = \delta^{\inv}$ and $H(t) = t^{A-1}$ for some $\delta, A > 0$, then the error term in \eqref{qwieq: sam} becomes $O\Big(e^{\beta x - \frac{\delta x}{A+1} }x^{\frac{(\rho - 1)A}{A+1}}\Big)$. This theorem is a special case of Theorem \ref{qwi: thspeccor} below where we consider more general singularities given by sums of pole-like and logarithmic terms. The implicit constant in \eqref{qwieq: sam} is allowed to depend on the function $\tau$ here and we warn the quick reader that in case $H(t)/t$ has subpolynomial growth, then a slightly better error term may be reached by considering $E_{\mathrm{AnA}}$, see Section \ref{qwi sec:mainth}. In Section \ref{qwisec: corol} we provide more general quantified versions of the Wiener-Ikehara theorem for the Laplace and Laplace-Stieltjes transforms.

The paper is organized as follows. In Section \ref{secqwi: mainth} we state the main theorem after discussing some preliminaries, the different types of boundary behavior for the Laplace transform and the Tauberian condition. Apart from the section on the Tauberian condition, these are essentially the same as in \cite{Deb-qik}, but we include them for the reader's convenience. In Section \ref{secqwi: proof} we prove the main theorem, which in view of the already established technology from \cite{Deb-qik}, amounts to establishing the Tauberian Lemma \ref{qwilem: taublem} for the new Tauberian condition. In the final Section \ref{qwisec: corol} we provide some corollaries of our main theorem that one can immediately appeal to in applications and we illustrate how to apply our theorem through the example of the prime number theorem.

Our notation is standard throughout. We use the Vinogradov notation $f(x) \ll g(x)$ or $f(x) = O(g(x))$, that is, there exists a constant $C$ such that $|f(x)| \leq Cg(x)$ for all $x$ in the domain of $f$ and $g$. If the implicit constant $C$ depends on external parameters, these shall always be incorporated into the notation. If we write $f(x) \ll_{M_n} g(x)$, then the implicit constant is understood to depend on the whole sequence $M_n$ instead of only on the $n$'th entry. We let $\chi_{E}$ denote the indicator function of a measurable set $E$, that is, $\chi_{E}(x) = 1$ if $x \in E$ and $0$ otherwise. The numbers $p,q \in [1,\infty]$ always represent \emph{conjugate} numbers, that is, $p^{-1} + q^{-1} = 1$.

The author wishes to thank Frederik Broucke and Jasson Vindas for valuable suggestions to previous versions of the manuscript.

\section{Statement of the main theorem} \label{secqwi: mainth}

\subsection{Preliminaries} \label{qiksecprelim}

If $U\subseteq\mathbb{R}$ is open, $\mathcal{D}(U)$ is the space of all smooth functions with compact support in $U$; its topological dual $\mathcal{D}'(U)$ is the space of distributions on $U$.
The standard Schwartz test function space of rapidly decreasing functions is denoted as usual by $\mathcal{S}(\mathbb{R})$, while $\mathcal{S}'(\mathbb{R})$ stands for the space of tempered distributions. The dual pairing between a distribution $f$ and a test function $\varphi$ is denoted as $\langle f, \varphi\rangle$, or  as $\langle f(x), \varphi(x)\rangle$ with the use of a dummy variable of evaluation. Locally integrable functions are regarded as distributions via $\langle f(x),\varphi(x)\rangle=\int_{-\infty}^{\infty}f(x)\varphi(x)\mathrm{d}x$. 

We fix the constants in the Fourier transform as
$\hat{\varphi}(t)=\mathcal{F}\{\varphi;t\}=\int_{-\infty}^{\infty}e^{-itx}\varphi(x)\:\mathrm{d}x.$ Naturally, the Fourier transform is well defined on $\mathcal{S}'(\mathbb{R})$ via duality, that is, the Fourier transform of $f\in\mathcal{S}'(\mathbb{R})$ is the tempered distribution $\hat{f}$ determined by $\langle \hat{f}(t),\varphi(t)\rangle=\langle f(x),\hat{\varphi}(x)\rangle$.
If $f\in\mathcal{S}'(\mathbb{R})$ has support in $[0,\infty)$, its Laplace transform is
$\mathcal{L}\left\{f;s\right\}=\left\langle f(u),e^{-su}\right\rangle,$ analytic on $\Re \:s>0$, and its Fourier transform $\hat{f}$ is the distributional boundary value of $\mathcal{L}\left\{f;s\right\}$ on $\Re \:s=0$. See \cite{vladimirov,zemanian} for complete accounts on Laplace transforms of distributions.

Let $M_n$ be a sequence of positive numbers. A smooth function $g : \R \rightarrow \C$ is said to be \emph{ultradifferentiable} at $t$ with respect to the sequence $M_{n}$ if there exist $B > 0$ and an open neighborhood $U$ of $t$ such that
\begin{equation} \label{eqriknnqu}
 \left|g^{(n)}(u)\right| \ll B^n M_{n}, \ \ \ \text{for all }n \in \N \text{ and } u \in U.
\end{equation}
A useful tool in this context is the so-called \emph{associated function}. This is
\begin{equation} \label{eqriknquass}                                                                                       
 M(x) = \sup_{n \in \mathbb{N}} \log\left(\frac{x^n}{M_{n}}\right), \ \ \ x > 0,
\end{equation}
whenever the supremum exists. For $M_{n} = n^n$ for instance, the associated function is $M(x) = x/e + O(1)$ as $x \rightarrow \infty$. We shall impose here that the sequence is always \emph{logarithmically convex}, that is, $M_n^2 \leq M_{n-1}M_{n+1}$ for all $n \geq 1$, and \emph{subanalytic}, that is, there exist $C,L > 0$ such that $M_n \geq CL^n n^n$ for all $n \in \N$.

\subsection{The hypotheses on the Laplace transform}

In this section we introduce the different types of boundary behavior we consider for the Laplace transform. These are the same assumptions that were discussed in \cite{Deb-qik}, but we recall them for convenience of the reader. 
Those assumptions will be exploited to estimate a certain integral in Lemma \ref{qwi: lemtransformcond}. Naturally, this is not an exhaustive list and, dependent of the application, there may be alternative behavior of the Laplace transform available that provides a better estimate of this integral, and therefore a better quantified rate in the Tauberian theorem. Nevertheless, with respect to the behavior cited below, our estimates in Lemma \ref{qwi: lemtransformcond} should often be sharp \cite[Section 5]{Deb-qik}.

Here $N$ denotes a natural number, $B, D$ and $ \delta_0$ are positive real numbers, $M_n$ is a logarithmically convex subanalytic sequence of positive numbers, $1 \leq p \leq \infty$---in particular we treat $L^{\infty}$-bounds for the Laplace transform---and $G,H$ are positive (locally integrable) functions defined on appropriate domains. Furthermore $\omega: \R_{+} \rightarrow \R_{+}$ is a positive non-decreasing function satisfying $\omega(y) \gg_{\omega} y$ as $y \rightarrow 0^{+}$.
 
We treat the following function classes. Depending on the class we sometimes impose additional constraints on $G$ and $H$.
  
\begin{enumerate}[label=(H.\arabic*)]
 \item $\mathcal{T}_{\mathrm{Dif}}(N,G,p,D)$ contains the functions $g: \R \rightarrow \C$ that are $N$ times \emph{differentiable} and satisfy $\|g^{(N)}(t)/G(t)\|_{L^{p}} \ll 1$ and $\max_{0 \leq j \leq N-1} |g^{(j)}(0)| \leq D$.   
\item $\mathcal{T}_{\mathrm{DifI}}(N,G,D)$ is comprised of the functions $g: \R \rightarrow \C$ that are $N$ times \emph{differentiable}, satisfy $\int^{R}_{-R}|g^{(N)}(t)| \mathrm{d}t \ll G(R)$ for all $R > 0$ and $\max_{0 \leq j \leq N-1} |g^{(j)}(0)| \leq D$. Here $G: [0,\infty) \rightarrow \R$ shall always be non-decreasing.
 \item $\mathcal{T}_{\mathrm{HC}}(N,G,p,\omega,D,\delta_{0})$ contains the functions $g: \R \rightarrow \C$ that are $N$ times differentiable, satisfy $\max_{0 \leq j \leq N} |g^{(j)}(0)| \leq D$ and for which $g^{(N)}$ is \emph{uniformly H\"older continuous} with respect to $L^{p}$, $G$ and $\omega$, that is 
 \begin{equation} \label{eqrfruhc}
 \left\|\frac{\sup_{\left|u\right| \leq \delta}\left|g^{(N)}(t+u) - g^{(N)}(t)\right|}{G(t)}\right\|_{L^{p}} \ll \omega(\delta), \quad 0 < \delta \leq\delta_{0},
\end{equation}
where the $L^{p}$-norm is taken with respect to the variable $t$. 
 \item $\mathcal{T}_{\mathrm{HCI}}(N,G,\omega,D,\delta_{0})$ comprises the functions $g: \R \rightarrow \C$ that are $N$ times differentiable, satisfy $\max_{0 \leq j \leq N} |g^{(j)}(0)| \leq D$ and for which
$$ \int^{R}_{-R} \sup_{|u| \leq \delta}|g^{(N)}(t+u) - g^{(N)}(t) | \mathrm{d}t \ll G(R) \omega(\delta), \ \ \ R > 0, \ \ \  0 < \delta \leq \delta_{0}.
$$
Here $G$ is non-decreasing.  

 \item $\mathcal{T}_{\mathrm{SA}}(M_{n},B,G,H,p)$ consists of the functions $g: \R \rightarrow \C$ that are infinitely differentiable and for which there exist locally integrable functions $G_{n}$ such that 
 \begin{equation} \label{equltradif}
 \left\| \frac{g^{(n)}(t)}{G_{n}(t)} \right\|_{L^{p}} \ll B^{n} M_{n}, \ \ \ n \in \N,
\end{equation}
with
\begin{equation} \label{eqaux}
  \left\| G_{n}\chi_{(-\lambda,\lambda)}\right\|_{L^{q}} \leq G(\lambda)^{n} H(\lambda), \ \ \ \text{for all }n \text{ and }  \lambda > 0.
\end{equation}
\item $\mathcal{T}_{\mathrm{SAI}}(M_{n},B,G,H)$ comprises the infinitely differentiable functions $g: \R \rightarrow \C$ satisfying
\begin{equation} \label{eqqiksai} \int^{R}_{-R} |g^{(n)}(t)| \mathrm{d}t \ll B^{n}M_{n}G(R)^{n} H(R), \ \ \ R > 0, \ \ \ n \in \mathbb{N}.
\end{equation}
Here $G$ and $H$ are non-decreasing. 
\item $\mathcal{T}_{\mathrm{An}}(G,H)$: this class contains the functions $g: \R \rightarrow \C$ that admit an \emph{analytic} extension to $\{z : |\Im z| \leq 1/G(|\Re z|)\}$ where it satisfies the bound $H(|z|)$. The function $G: \mathbb{R}_{+} \rightarrow \mathbb{R}_{+} $ is continuous and non-decreasing with $G(0) \neq 0$. The function $tH(t)$ is non-decreasing on $\mathbb{R}_{+}$. 
\end{enumerate}
If all classes are considered at once, we use the notation $\mathcal{T}_\ast$. 

\subsection{The Tauberian condition} \label{secqwi: taubcond}

Let $F,F_\mathrm{D}, h, h_\mathrm{D}$ be positive real-valued functions with domain $\mathbb{R}$, $k \geq 1$ be a positive real-valued function defined on the positive half-axis and $X, C_j$ with $1 \leq j \leq 8$ be positive real numbers and $0 \leq \alpha < 1$. 

We shall consider the following class of functions $\eT(F, F_\mathrm{D}, h, h_\mathrm{D},k, X, C_j, \alpha)$. We say that the functions $S: \mathbb{R} \rightarrow \mathbb{R}$ that belong to this class satisfy the \emph{Tauberian condition}.
\begin{enumerate}[label=(T.\arabic*)]
 \item \label{qwi hypoLapl} $S$ is identically $0$ on the negative half-axis,
\item \label{qwi taco1} $h(x)(S(x) + F(x))$ is non-decreasing for $x \geq X$,
 \item \label{qwi taco1bis} $h(X) (S(X) + F(X)) \geq \sup_{x \leq X} h(x)(S(x) + F(x))$,
 \item  \label{qwi taco3}$\frac{|h(x+y) - h(x)|}{h(x+y)} \leq C_1 y \frac{h_{\mathrm{D}}(x)}{h(x)} \exp(y^{\alpha})$ for $y > 0$ and $x \geq X$,
 \item \label{qwi taco4} $F(x+y) |h(x+y) - h(x) | \leq C_2 F(x) |y| h_{\mathrm{D}}(x) \exp(|y|^{\alpha})$ for $y < 0$ and $x \geq X$,
 \item \label{qwi taco5}$|F(x+y) - F(x)| \leq C_3 F_{\mathrm{D}}(x) |y| \exp(|y|^{\alpha})$, for all $y \in \R$, $x\geq X$,
 \item \label{qwi taco6}$|S(x) |\leq  C_{4} F(x)$ for all $x \leq X$,
 \item \label{qwi taco7}$h(X) |S(X) + F(X)| \leq C_{5}$,
 \item \label{qwi taco8}$h(x) F(x) \geq C_{6} > 0$, for $x \geq X$,
 \item \label{qwi taco9}$h(x+u) \leq C_{7} h(x) k(u)$ for all $u \geq 0$ and $x \geq X$,
 \item \label{qwi auxhypo9} $F_\mathrm{D}(x) \leq C_{8} F(x)$ for all $x \geq X$.
 \end{enumerate}

The functions $F_\mathrm{D}$ and $h_\mathrm{D}$ should be interpreted as the derivatives of $F$ and $h$ respectively. The key assumption is \ref{qwi taco1}. The rest are technical assumptions of which the most important ones are \ref{qwi taco3}, \ref{qwi taco4} and \ref{qwi taco5} describing the most significant properties for $F$ and $h$. The first assumption \ref{qwi hypoLapl} is imposed because we work with Laplace transforms. The other assumptions are fairly minor and mainly serve to achieve an effective estimate. The function $k$ is not too important; it only plays a role in the minor condition \ref{qwi taco9}, which, if need be, could even be relaxed. In the classical Wiener-Ikehara setting where $h(x) = e^{x}$, one may choose $k(u) = e^{u}$.

Our hypotheses, particularly \ref{qwi taco3} and \ref{qwi taco4}, have been designed specifically to hold for the classical Wiener-Ikehara choice of $h(x) = e^{x}$, while \ref{qwi taco5} occurs already in \cite{Deb-qik} and represents essentially a subexponential growth condition for $F$; we refer to \cite[Section 2.2]{Deb-qik} for some examples of appropriate functions $F$. In fact, if $h(x) = \exp(\beta x)$ and $h_{\mathrm{D}}(x) = \beta \exp(\beta x)$ with $\beta > 0$, then \ref{qwi taco3} is automatically satisfied with $C_1 = 1$ and \ref{qwi taco4} follows from \ref{qwi taco5} and \ref{qwi auxhypo9} with $C_2 = 1 + C_3C_8/\beta$. We also remark the properties \ref{qwi taco3} and \ref{qwi taco4} hold trivially if $h(x) = 1$. 

The reason our more direct treatment of the Tauberian condition from \cite{Deb-qik} fails here is because $e^{x}$ is not a Denjoy-Carleman non-quasianalytic weight function, that is, there is no non-trivial function $\phi \in L^{1}(\mathbb{R})$ whose Fourier transform is compactly supported and which decays as $\phi(x) \ll e^{-|x|}$. The hypotheses \ref{qwi taco3} and \ref{qwi taco4} are customized to enable a successful, more technical alternative approach. 

We also observe that if \ref{qwi taco4} and \ref{qwi taco5} hold for functions $F_1$ and $F_2$ playing the roles of $F$ (and $F_{1,\mathrm{D}}$ and $F_{2,\mathrm{D}}$ in the role of $F_\mathrm{D}$), then \ref{qwi taco4} and \ref{qwi taco5} also hold for $F = \max\{F_1, F_2\}$ (and $F_\mathrm{D}= \max\{F_{1,\mathrm{D}},F_{2,\mathrm{D}}\}$). Such a modification of the function $F$ may be useful in Theorem \ref{qwi: thmain} below to fulfill the assumption $  E_{\ast}(x,1) \ll F(x)$; see the definition of $E_{\ast}$ below Theorem \ref{qwi: thmain}.

The Tauberian condition $\eT$ is designed to encode non-decreasing functions $\tau$ in Theorem \ref{qwi: thold}, but in Remark \ref{qwi: remtaubcond} we discuss a possible generalization that can be considered as some quantified analogue of the well-known slowly decreasing Tauberian condition.

\subsection{The main theorem} \label{qwi sec:mainth}

The main theorem of this paper reads as follows.

\begin{theorem} \label{qwi: thmain} Suppose that $S$ belongs to the class $\eT(F, F_\mathrm{D}, h, h_\mathrm{D}, k, X, C_j, \alpha)$ and that 
$$\mathcal{L}\{S;s\} = \int^{\infty}_{0} e^{-sx} S(x) \mathrm{d}x
$$
converges on $\Re s > 0$ and admits (ultra)distributional\footnote{Ultradistributional boundary values in this context means that \eqref{eqqwi: distbound} only need to hold for all $\phi \in \mathcal{D}(\R)$ whose Fourier transform decays as $\hat{\phi}(t) \ll \exp(-|x|^{\gamma})$ for some fixed $\gamma < 1$. We remark that distributional boundary behavior is automatically fulfilled if the Laplace transform admits a continuous, analytic or $L^p$-extension.} boundary values on the line $\Re s = 0$ as $\mathcal{L}\{S; it\}= g(t)$, that is,
\begin{equation} \label{eqqwi: distbound}
 \lim_{\sigma \rightarrow 0+} \int^{\infty}_{-\infty} \mathcal{L}\{S; \sigma + it\} \phi(t) \dif t = \int^{\infty}_{-\infty} g(t) \phi(t) \dif t, \ \ \ \text{ for all } \phi \in \mathcal{D}(\mathbb{R}),
\end{equation}
and $g$ is a function that belongs to $\mathcal{T}_{\ast}$.

Then, for\footnote{The range of validity for \eqref{eqqwi: mainresultaux} and \eqref{eqqwi: mainresult} is $x \geq \max\{X,2\pi, \pi/\delta_0\}$ in case $\mathcal{T}_{\ast} = \mathcal{T}_{\mathrm{HC}}$ or $\mathcal{T}_{\ast} = \mathcal{T}_{\mathrm{HCI}}$.} $x \geq X$, 
\begin{equation} \label{eqqwi: mainresultaux}
 |S(x)| \ll E_{\ast}(x,1) + F(x).
\end{equation}
Moreover, if $E_{\ast}(x,1) \ll F(x)$, then the estimate
\begin{equation} \label{eqqwi: mainresult}
  |S(x)| \ll \inf_{\lambda \geq 1} \left\{E_{\ast}(x,\lambda) + \frac{h_\mathrm{D}(x)}{h(x)} \frac{F(x)}{\lambda} + \frac{F_\mathrm{D}(x)}{\lambda}\right\} 
\end{equation}
 holds for $x \geq X$.

The implicit constants in \eqref{eqqwi: mainresultaux} and \eqref{eqqwi: mainresult} depend only on the variables $C_j, X, \alpha$ of the Tauberian condition and the implicit constant in $E(x,1) \ll F(x)$. The functions $E_\ast$ are given below. 
\end{theorem}

The constants $C$ below in the definition of $E_\ast(x,\lambda)$ depend only on the parameters in their subscript. One may in principal track down an explicit description of the constants by going through the calculations in \cite[Section 4]{Deb-qik}, see Lemma \ref{qwi: lemtransformcond} below.

\begin{itemize}
 \item $E_{\mathrm{Dif}}(N,G,p,D,x,\lambda) = C_{N,D} x^{-N} (1 + \|G \chi_{(-\lambda,\lambda)}\|_{L^{q}})$. 
 \item $E_{\mathrm{DifI}}(N,G,D,x,\lambda) = C_{N,D} x^{-N} (1 + G(\lambda))$.
 \item $E_{\mathrm{HC}}(N,G,p,\omega,D,\delta_{0},x,\lambda) = C_{N,\omega,D,\delta_{0}} x^{-N} \omega(\pi/x) (1 +\|G \chi_{(-\lambda,\lambda)}\|_{L^{q}})$. 
 \item $E_{\mathrm{HCI}}(N,G,\omega,D,\delta_{0},x,\lambda) = C_{N,\omega,D,\delta_{0}} x^{-N} \omega(\pi/x) (1+ G(\lambda))$.
\item $E_{\mathrm{SA}}(M_{n},B,G,H,p,x,\lambda) = C_{M_{n}}  \exp\left(-M\left(\frac{x}{B G(\lambda) + \kappa\lambda^{-1}}\right)\right)H(\lambda)$, for some real constant $\kappa$ depending\footnote{An admissible value for $\kappa$ in $E_{\mathrm{SA}}$ and $E_{\mathrm{SAI}}$ is $\kappa = 124/L$ where $L$ is such that $M_n \gg L^n n^n$.} only on $M_{n}$ and where $M$ is the associated function of the sequence $M_{n}$. 
\item $E_{\mathrm{SAI}}(M_{n},B,G,H,x,\lambda) = C_{M_{n}} \exp\left(-M\left(\frac{x}{BG(\lambda) +\kappa\lambda^{-1}}\right)\right)H(\lambda)$, for some real constant $\kappa$ depending only on $M_{n}$. 
\item $E_{\mathrm{An}}(G,H,x,\lambda) =  C \exp\left(-\frac{x}{G(\lambda)+ \kappa \lambda^{-1}}\right) \left\{\int^{\lambda}_{0}H(t)\mathrm{d}t + G(0)\inv \displaystyle{\max_{t \leq 5G(0)\inv}} H(t)\right\} \allowbreak \ll_{G(0),H} C \exp\left(-\frac{x}{G(\lambda)+\kappa\lambda^{-1}}\right)\int^{\lambda}_{0}H(t)\mathrm{d}t$ if $\lambda \geq 2 G(0)\inv$ and $E_{\mathrm{An}}(G,H,x,\lambda) = \infty$ otherwise. An admissible value for $\kappa$ here is $675$.
\end{itemize} \par

We ought to mention that in case of $\mathcal{T}_{\mathrm{An}}$ and if $H(t)/t$ is of subpolynomial growth, some slight improvement of \eqref{eqqwi: mainresult} may be possible; this should occur when the estimate $\int^{\lambda}_{0} G(t)^n H(t) \dif t \ll G(\lambda)^n \int^{\lambda}_0 H(t) \dif t$ has a relevant improvement \cite[Remark 4.4]{Deb-qik}. For example, if $G(\lambda)/\log^{\gamma} \lambda$ is eventually non-decreasing as $\lambda \rightarrow \infty$ for some $\gamma > 0$, then $E_{\mathrm{An}}$ in \eqref{eqqwi: mainresult} may be replaced with $E_{\mathrm{AnA}}$, where $E_{\mathrm{AnA}}(x,\lambda) = \infty$ if $\lambda \leq 2G(0)^{-1}$ and 
\begin{equation} \label{qwi: eqmodifiedean}
E_{\mathrm{AnA}}(G,H,S,x,\lambda) = C_{G,H,S} \inf_{n \in \N} \left\{\frac {n^n G(\lambda)^{n} \left(1 + 2\kappa G(\lambda)^{-1}\lambda^{-1}\right)^n }{e^n x^n }\left(1 + \frac{\log \lambda}{\gamma n+1}\right) \lambda H(\lambda)\right\}
\end{equation}
otherwise. In the model quantified Wiener-Ikehara theorem \cite[Th. 1.4]{Deb-qik} where one handles only a simple pole, this modification has already been incorporated. In case $M = K =\log$ (in the formulation of \cite{Deb-qik}) this actually leads to an improvement of the quantified rate in the Tauberian theorem.

\section{Proof of the main theorem} \label{secqwi: proof}

We start our investigations with the deduction of \eqref{eqqwi: mainresultaux}. The result has additional intrinsic interest as it can provide an estimate if the good behavior of the Laplace transform is given only on the interval $(1-i\lambda, 1 + i \lambda)$ for a fixed $\lambda > 0$. In this paper, however, it mainly serves as a preliminary (rough) estimate, that will help us establish \eqref{eqqwi: mainresult}.

\begin{lemma} \label{qwi: taublem} Let $\phi \in L^{1}(\mathbb{R})$ be a positive function such that
\begin{enumerate}[label=$(\alph*)$]
 \item $\int^{\infty}_{0} \frac{\phi(u)}{k(u)} \mathrm{d}u =: C_{9}$ is finite,
 \item \label{qwi boundcondphi2} $\int^{\infty}_{-\infty} |y| \exp(|y|^\alpha) |\phi(y)| \mathrm{d}y \leq C_{10}$,
 \item \label{qwi boundcondphi3} $\int^{\infty}_{-\infty}\phi(u) \mathrm{d}u \leq C_{11}$.
\end{enumerate}
If $S$ belongs to the class $\eT$, then there exist constants $C_{12}$ and $C_{13}$ such that, for $x \geq X$, 
\begin{equation} \label{eqqwi: boundedness} |S(x)| \leq C_{12} \left|\int^{\infty}_{-\infty} S(x + u) \phi(u)\mathrm{d}u\right| + C_{13} F(x).
\end{equation}
Admissible constants are $C_{12} = C_{7} C_{9}^{-1}$ and $C_{13} = \max\{C_{4}, C_{5}C_{6}^{-1} + 1, C_{7}C_{4}C_{9}^{-1}(C_{11} + C_3 C_{10} C_{8}),C_{7}C_{9}^{-1}(C_{5} C_{6}^{-1}+ 1)(C_{11} + C_3 C_{10} C_{8})  \}$.
\end{lemma}
\begin{proof} 

Since $h(x) (S(x) + F(x))$ is non-decreasing for $x \geq X$, we have by \ref{qwi taco9} that
\begin{align*}
 S(x) + F(x) & = \frac{1}{C_{9}} \int^{\infty}_{0} \frac{\Big(S(x) + F(x)\Big) h(x) \phi(u)}{h(x) k(u)}\mathrm{d}u \\ 
 & \leq \frac{1}{C_{9}} \int^{\infty}_{0} \frac{ h(x+ u)\Big( S(x + u) + F(x+u)\Big) \phi(u)}{h(x)k(u)}\mathrm{d}u \\
 & \leq \frac{C_{7}}{C_{9}} \int^{\infty}_{0} \Big( S(x + u) + F(x+u)\Big ) \phi(u)\mathrm{d}u. 
\end{align*}

Now we wish to extend the domain of integration of the last integral to the full real axis. For this, we search 
for a constant $C_{14}$ that turns $S(x) + C_{14}F(x)$ into a non-negative function. Since $h(x)(S(x) + F(x))$ is non-decreasing for $x \geq X$, we find, using also \ref{qwi taco7} and \ref{qwi taco8}, that 
$$ S(x) \geq \frac{h(X) \big(S(X) + F(X)\big) }{h(x)} - F(x) \geq -(C_{5} C_{6}^{-1} + 1) F(x), \ \ \ x\geq X.
$$

Setting $C_{14} := \max\{C_{4}, C_{5}C_{6}^{-1} + 1\}$ and using \ref{qwi taco6} we thus have that $S(x) + C_{14} F(x)$ is non-negative for all real $x$. This also immediately implies the desired lower bound for $S(x)$ in \eqref{eqqwi: boundedness}.

Therefore, by the positivity of $\phi$,
\begin{align*}
 \int^{\infty}_{0} \Big( S(x + u) + F(x+u)\Big ) \phi(u)\mathrm{d}u & \leq \int^{\infty}_{-\infty} \Big( S(x + u) + C_{14} F(x+u)\Big ) \phi(u)\mathrm{d}u \\
 & \qquad - (C_{14} - 1) \int^{\infty}_{0} F(x+u) \phi(u)\mathrm{d}u \\
 & \leq \int^{\infty}_{-\infty} S(x + u)  \phi(u)\mathrm{d}u + C_{14} \int^{\infty}_{-\infty}  F(x+u) \phi(u)\mathrm{d}u.
 \end{align*}
Now, \ref{qwi taco5}, \ref{qwi boundcondphi2} and \ref{qwi boundcondphi3} give us
$$\int^{\infty}_{-\infty}  F(x+u) \phi(u)\mathrm{d}u \leq C_{11} F(x) + C_3 C_{10} F_{\mathrm{D}}(x),
$$
which together with \ref{qwi auxhypo9} completes the proof of the lemma.
\end{proof}

When $h(x) =e^{\beta x}$ with $\beta > 0$, this is essentially the classical boundedness result for the Wiener-Ikehara theorem \cite[Th. 3.1]{Debruyne-VindasComplexTauberians}. 

The integral $\int^{\infty}_{-\infty} S(x + u) \phi(u)\mathrm{d}u$ can then be estimated by $E_{\ast}(x,1)$, see Lemma \ref{qwi: lemtransformcond}. Thus, if $E_{\ast}(x,1) \ll F(x)$, Lemma \ref{qwi: taublem} and \ref{qwi taco6} deliver us a constant $C_{15}$, depending in principle on $\alpha$, $C_j$ and the implicit constant in $E_{\ast}(x,1) \ll F(x)$, such that the initial boundedness result
\begin{equation}
\label{qwi taco2} |S(x)| \leq C_{15} F(x),
\end{equation}
holds for all real $x$. 

We now come to the main \emph{Tauberian lemma} for the flexible quantified Wiener-Ikehara theorem. We will use \eqref{qwi taco2} as a preliminary auxiliary estimate. This lemma will complete the treatement of the Tauberian condition for Theorem \ref{qwi: thmain}.

\begin{lemma} \label{qwilem: taublem} Let $S : \mathbb{\R} \rightarrow \R$ be a function satisfying the Tauberian condition $\eT$ and \eqref{qwi taco2} for all real $x$. 

Suppose that $\phi: \mathbb{R} \rightarrow \mathbb{R}$ is such that 
\begin{enumerate}[label=$(\alph*)$]
\item $\int^{\infty}_{-\infty} \phi(y) \mathrm{d}y = 1$, 
\item $\int^{\infty}_{-\infty} \left|y| \exp(|y|^\alpha) |\phi(y)\right| \mathrm{d}y \leq C_{10}$,
\item \label{qwi: taublemhypo3} $\phi(y) \geq 0$ for $y \geq 0$ and $\phi(y) \leq 0$ for $y \leq 0$.
\end{enumerate}

Then, there is a constant $C_{16}$ such that, for $x \geq X$,
\begin{equation}\label{eqqwi: restaublem}
\begin{aligned}
&  \inf_{\lambda \geq 1}\left\{-C_{16} \left(\frac{h_\mathrm{D}(x)}{h(x)} \frac{F(x)}{\lambda} + \frac{F_\mathrm{D}(x)}{\lambda}\right) + \lambda \int^{\infty}_{-\infty} S(x + y) \phi(-\lambda y)  \dif y\right\} \\
 & \qquad \leq  S(x) \leq \inf_{\lambda \geq 1}\left\{C_{16} \left(\frac{h_\mathrm{D}(x)}{h(x)} \frac{F(x)}{\lambda} + \frac{F_{\mathrm{D}}(x)}{\lambda}\right) + \lambda \int^{\infty}_{-\infty} S(x + y) \phi(\lambda y)  \dif y\right\}.
\end{aligned}
\end{equation}
An admissible value is $C_{16} = \max\{ C_{10}C_1(C_{15}+1), C_{10}C_2(C_{15}+1),C_{10} C_3 \}$.
\end{lemma}
\begin{proof} We need to establish some inequalities between $S(x)$ and $S(x+y)$ to relate $S(x)$ with the convolution $S \ast \phi$. We always let $x \geq X$ throughout this proof. For $y > 0$ we use \ref{qwi taco1} to write
\begin{align*}
 S(x) & = \frac{h(x)}{h(x+y)} \Big(S(x) + F(x)\Big) - \frac{h(x)}{h(x+y)}F(x) + \bigg(1 - \frac{h(x)}{h(x+y)}\bigg) S(x) \\
  & \leq S(x+y) + F(x+y) - F(x) +  \bigg(1 - \frac{h(x)}{h(x+y)}\bigg) \bigg(S(x) + F(x) \bigg). 
\end{align*}
Therefore, by \eqref{qwi taco2}, \ref{qwi taco3} and \ref{qwi taco5}, we obtain for $y > 0$ that
\begin{equation}
 S(x) - S(x+y) \leq y \exp(|y|^{\alpha}) \bigg\{C_3 F_{\mathrm{D}}(x) + C_1(C_{15}+1) \frac{h_{\mathrm{D}}(x)}{h(x)} F(x) \bigg\}.
\end{equation}

For $y \leq 0$, we use \ref{qwi taco1} and \ref{qwi taco1bis} to find
\begin{align*}
 S(x) & = \frac{h(x)}{h(x)} \Big(S(x) + F(x)\Big) - F(x) \geq \frac{h(x+y)}{h(x)}\Big(S(x+y) + F(x+y)\Big) - F(x)  \\
  & = S(x+y) + F(x+y) - F(x) + \bigg(S(x+y) + F(x+y)\bigg) \bigg(\frac{h(x+y)}{h(x)} -1 \bigg) . 
\end{align*}
So, using \eqref{qwi taco2}, \ref{qwi taco4} and \ref{qwi taco5}, we achieve for all $y\leq0$ the inequality
\begin{equation}
S(x) - S(x+y) \geq -  |y| \exp(|y|^{\alpha}) \bigg\{C_3 F_{\mathrm{D}}(x) + C_2(C_{15}+1) \frac{h_{\mathrm{D}}(x)}{h(x)} F(x)\bigg\}.
\end{equation}

Now we are in a position to apply the argument from \cite[Lemma 3.1]{Deb-qik}. Using the properties of $\phi$, we find
\begin{align*}
 S(x) & = \lambda \int^{\infty}_{-\infty} S(x) \phi(\lambda y) \dif y \leq \lambda \int^{\infty}_{-\infty} S(x+ y) \phi(\lambda y) \dif y \\
 & \qquad +  \bigg\{C_3 F_{\mathrm{D}}(x) + \max\{C_1,C_2\}(C_{15}+1) \frac{h_{\mathrm{D}}(x)}{h(x)} F(x)  \bigg\}  \int^{\infty}_{-\infty} \lambda y \exp(|y|^{\alpha}) \phi(\lambda y) \dif y
\end{align*}
For $\lambda \geq 1$, the integral in the last expression is bounded by
\begin{equation*}
 \int^{\infty}_{-\infty} \lambda y \exp(|y|^{\alpha}) \phi(\lambda y) \dif y \leq \frac{C_{10}}{\lambda}.
\end{equation*}
This delivers the upper bound for $S$ in \eqref{eqqwi: restaublem}. The lower bound is derived by taking $\phi(-y)$ instead of $\phi(y)$ in the above argument.
\end{proof}

The proof of the main theorem is now complete if we estimate $\int^{\infty}_{-\infty} S(x + y) \phi(\lambda y)  \dif y$ with the information from $\mathcal{T}_\ast$. This was already done in \cite[Section 4]{Deb-qik} and we collect it in the following lemma.

\begin{lemma} \label{qwi: lemtransformcond}Suppose that $\mathcal{L}\{S; it\} = g(t) \in \mathcal{T}_{\ast}$, $ \operatorname*{supp} S \subseteq [0,\infty)$ and $S(x) + F(x)$ is eventually\footnote{This means that there exists $Y$ such that either $S(x) + F(x) \geq 0$ for all $x \geq Y$ or $S(x) + F(x) \leq 0$ for all $x \geq Y$.} of one sign, where $F(x) \ll \exp(x^{\alpha})$ as $x \rightarrow \infty$ for some $\alpha < 1$. There exist constants $C_9, C_{10}, C_{11}$, depending only on $\alpha$ and an absolute constant $C_{17}$, such that for each\footnote{In case $\mathcal{T}_{\ast} = \mathcal{T}_{\mathrm{HC}}$ or $\mathcal{T}_{\ast} = \mathcal{T}_{\mathrm{HCI}}$ the range for $x$ becomes $x \geq \max\{2\pi, \pi/\delta_0\}$.} $x > 0$ and $\lambda \geq 1$, there exists $\phi = \phi_{x,\lambda} \in \mathcal{F}(\mathcal{D}(\mathbb{R}))$ such that
\begin{equation} \label{qwi: resulttransformcond}
  \lambda \bigg|  \int^{\infty}_{-\infty} S(x + y) \phi(\lambda y)  \dif y\bigg| \leq C_{17} E_{\ast}(x,\lambda), 
\end{equation}
and the hypotheses from either Lemma \ref{qwi: taublem} or Lemma \ref{qwilem: taublem} for $\phi$ are fulfilled. One may also assume additionally that $\phi(x) \ll_{\gamma} \exp(-|x|^{\gamma})$ for any fixed\footnote{The choice of $\phi$ is then dependent of $\gamma$.} $\gamma < 1$.

One may also take $\phi(-x)$ instead of $\phi(x)$ in \eqref{qwi: resulttransformcond}.
\end{lemma}

The result \eqref{qwi: resulttransformcond} with $E_{\mathrm{AnA}}$ being defined as \eqref{qwi: eqmodifiedean}, can be shown with the arguments in the proof of \cite[Th. 1.3]{Deb-qik}. 

We conclude this section with some final comments on the Tauberian condition.

\begin{remark} \label{qwi: remtaubcond} In the literature of unquantified Tauberian theorems, the Tauberian condition of \emph{slow decrease} or \emph{slow oscillation} is a standard generalization for a non-decreasing function; it has the benefit that certain singularities of the Laplace transform can be subtracted without violating the Tauberian condition and that the functions may temporarily possess some mild decrease.

In a quantified setting, considering such a fixed Tauberian condition is less advisable because it may lay some undesirable restrictions on the obtained rate in the Tauberian theorem; in the Ingham-Karamata theorem for example it would be impossible to get any quantified rate under the Tauberian condition of slow decrease as any $o(1)$-function is slowly decreasing. An additional undesirable feature would be that some reasonable singularities could not be handled, but this issue also occurs in the unquantified setting and this was the main reason for the introduction of flexible Tauberian conditions in \cite{Deb-qik}. 

Nevertheless, although we believe it is the best practice to encode the significant singularities of the Laplace transform separately in a flexible function $F$ rather than directly in the function $\tau$ from Theorem \ref{qwi: thold}, we acknowledge it may sometimes still be useful to allow some controlled decrease of the function $\tau$ from Theorem \ref{qwi: thold}.

Therefore, we observe that one may replace \ref{qwi taco1} and \ref{qwi taco1bis} in the Tauberian condition $\eT$ with
\begin{equation} \label{qwi: quananslowdec}
\begin{aligned}
h(y)(S(y) + F(y)) & \leq h(x)(S(x) + F(x)) + h(x) E(x), & y \leq x,\ \  x \geq X, \\
h(y)(S(y) + F(y)) & \geq h(x)(S(x) + F(x)) - h(y) E(x), & y \geq x \geq X,
\end{aligned}
\end{equation}
where $E: [X,\infty) \rightarrow (0,\infty)$ is a positive function. This can be considered as some quantified analogue of slow decrease. Naturally the function $E$ affects the result in Theorem \ref{qwi: thmain} but the changes in the proofs are straight-forward and the theorem remains true as long as $E(x) \ll $ the right-hand side of respectively \eqref{eqqwi: mainresultaux} and \eqref{eqqwi: mainresult}. Inserting the necessary modifications one needs the extra term $2C_7C_9^{-1} E(x)$ in the right-hand side of \eqref{eqqwi: boundedness} in Lemma \ref{qwi: taublem} and the extra term $E(x) \int^{\infty}_{-\infty} |\phi|$, with appropriate signs, on the left-most and right-most expression of \eqref{eqqwi: restaublem} in Lemma \ref{qwilem: taublem}. (The functions $\phi_{x,\lambda}$ in Lemma \ref{qwi: lemtransformcond} can be constructed in such a way that $\int^{\infty}_{-\infty} |\phi_{x,\lambda}| \leq C_{18}$, with $C_{18}$ a constant depending only on $\alpha$.)

Often it even suffices to only have \eqref{qwi: quananslowdec} for $y$ in a neighborhood of $x$ (of fixed length). Namely, when, in the proofs of the Tauberian lemmas, the inequality \eqref{qwi: quananslowdec} is applied inside the integral, the mass of the integral $\int^{\infty}_{-\infty} \phi( u) \dif u$ is concentrated near $u = 0$ and thus there should be not that much wiggle room to relax the \eqref{qwi: quananslowdec} if $y$ is near $x$. However, when $u$ is relatively large, one could leverage the decay of $\phi$ near $\infty$ to ensure that worse estimates than \eqref{qwi: quananslowdec} when $y$ is far away from $x$, obtained either through iteration or simply available, remain acceptable after one inserts these inside the integral. With this technique, one can for instance show the unquantified Wiener-Ikehara theorem under the \emph{log-linear slowly decreasing} Tauberian condition \cite[Th. 1]{Debruyne-VindasGeneralWI}, see also \cite{zhang-exact}, which strictly speaking is not captured by \eqref{qwi: quananslowdec}. (The proof in \cite{Debruyne-VindasGeneralWI} differs from the approach proposed here.)

\end{remark}

\section{Corollaries} \label{qwisec: corol}

In this final section, we provide some corollaries involving pole-like (and other) singularities on the Laplace transform  of our main theorem. These versions are somewhat easier to use in applications.

We say that a function $\tau: [0,\infty) \rightarrow \mathbb{R}$ is \emph{eventually non-decreasing} if there exists $X > 0$ such that $\tau$ is non-decreasing on $[X,\infty)$.

\begin{theorem} \label{qwi th:cor1} Let $\tau: [0,\infty) \rightarrow \mathbb{R}$  be a locally integrable eventually non-decreasing function, $\beta > 0$ and $F: \R \rightarrow \R$ be a non-negative locally integrable function such that 
\begin{equation} \label{eqqwi: hypFcor}
 |F(x+y) - F(x)| \leq C F(x) |y| \exp(|y|^{\alpha}), \ \ \  \text{for all } y \in \R, \: x \geq X, 
\end{equation} 
for some $C, X > 0$, $0 \leq \alpha < 1$ and $F \not \equiv 0$.

Then, if 
\begin{equation}
 \mathcal{L}\{\tau, s + \beta\} - \mathcal{L}\{F;s\} = \int^{\infty}_{0} \tau(x) e^{-(s+\beta)x} \dif x - \int^{\infty}_{0} F(x) e^{-sx} \dif x,
\end{equation}
converges for $\Re s > 0$ and admits a continuous\footnote{It is also allowed that the Laplace transform admits (ultra)distributional boundary values instead.} extension to $\Re s = 0$ with $g(t) := \mathcal{L}\{\tau; \beta + it\} -  \mathcal{L}\{F; it\}  \in \mathcal{T}_\ast$, then 
\begin{equation} 
 \tau(x) \ll e^{\beta x} \Big(F(x) + E_{\ast}(x,1)\Big), \ \ \ x \rightarrow \infty.
\end{equation}

Furthermore, if, additionally, $E_{\ast}(x,1) \ll F(x)$, then
\begin{equation} \label{qwi: thusefulmain}
 \tau(x) = e^{\beta x} F(x) + O\left(e^{\beta x} \inf_{\lambda \geq 1} \left\{E_{\ast}(x,\lambda) + \frac{F(x)}{\lambda}\right\}  \right), \ \ \ x \rightarrow \infty.
\end{equation}
\end{theorem}
The $O$-constants here are allowed to depend on the function $\tau$.

\begin{proof} The proof is essentially just applying Theorem \ref{qwi: thmain} to $S(x) := e^{-\beta x} \tau(x) - F(x)$ and noting that the critical property of the Tauberian condition \ref{qwi taco1} with $h(x) = e^{\beta x}$ follows from the fact that $\tau$ is eventually non-decreasing. Yet, since we wish to provide an easy-to-apply theorem without unnecessary extra assumptions on $\tau$ or $F$ that are annoying to verify, we still need to do some reductions and a little technical work to check that the other properties of $\eT$ are not troublesome.

First of all, we observe that $F \not \equiv 0$ and \eqref{eqqwi: hypFcor} imply that $F(x) \gg \exp(-x^{\gamma})$ for sufficiently large $x$, say $x \geq X$, and for some $\gamma < 1$. We may assume without loss of generality that $F \equiv 0$ on $(-\infty,0)$ as \eqref{eqqwi: hypFcor} clearly holds for $F(x) 1_{[0,\infty)}(x)$; replace if necessary $X$ to ensure $X \geq 1/C$. The assumption $F(x) 1_{[0,\infty)}(x) \not \equiv 0$ follows from $F(x) \gg \exp(-x^{\gamma})$ for $x \geq X$.

Next, say that $\tau$ is non-decreasing on $[X,\infty)$. We argue that we may assume without loss of generality that $\tau$ is non-negative on $[X,\infty)$. Indeed, one may, if necessary, consider $\tilde{\tau}(x)= \tau(X) + \tau(x)$ and $\tilde{F}(x) = F(x) + 1_{[0,\infty)}(x) e^{-\beta x} \tau(X)$. Since $F(x) \gg \exp(-x^{\gamma})$ for $x \geq X$, we have that $\tilde{F} \asymp F$,  $\operatorname*{supp} \tilde{F} \subseteq [0,\infty)$ and that $\tilde{F}$ also satisfies \eqref{eqqwi: hypFcor}, possibly with a different $C$ and $\alpha$. Finally, we redefine $\tau$ on $[0,X)$ where we set it $\tau(x) =0$; obviously this does not affect the asymptotic of $\tau$ as $x \rightarrow \infty$, but it has some consequences for the Laplace transform of $\tau$, which we take into account later in the proof.

We set $S(x) := e^{-\beta x} \tau(x) - F(x)$. We have to verify that $S$ satisfies the Tauberian condition $\eT$ after selecting $h(x) = \exp(\beta x)$, $h_{\mathrm{D}}(x) = \beta \exp(\beta x)$, $F_\mathrm{D}(x) = F(x)$, $k(u) \equiv e^{\beta u}$ and suitably chosen constants $C_j$, $1 \leq j \leq 8$. It is clear that $S$ vanishes on the negative half-axis. The non-decrease of $\tau$ on $[X,\infty)$ ensures the critical property \ref{qwi taco1}, while the non-negativity of $\tau(X)$ and the fact that we redefined $\tau$ on $[0,X)$ as the zero function provide \ref{qwi taco1bis}. Property \ref{qwi taco5} follows from \eqref{eqqwi: hypFcor} and \ref{qwi auxhypo9} is obvious. As discussed in Section \ref{secqwi: taubcond}, these two properties then imply \ref{qwi taco3} and \ref{qwi taco4} with $h(x) = \exp(\beta x)$ and $h_{\mathrm{D}}(x) = \beta \exp(\beta x)$. Property \ref{qwi taco6} follows from the definition of $S$ and the fact that we defined $\tau$ to be $0$ on $[0,X)$. One can deduce \ref{qwi taco8} from $F(x) \gg \exp(-x^{\gamma})$ and the remaining properties are trivially verified for suitably chosen constant $C_j$.

Now one has that $\mathcal{L}\{S;s\} = \mathcal{L}\{\tau; s+\beta\} - \mathcal{L}\{F;s\}$, but this Laplace transform differs from that in the statement of the theorem by $\mathcal{L}\{W;s\}$, where $W$ is a compactly supported locally integrable function, because we have redefined $\tau$ on a compact interval. This means that $\mathcal{L}\{S;it\}$ may not necessarily belong to $\mathcal{T}_\ast$ and this extra function $W$ necessitates a little adjustment in the proof of Theorem \ref{qwi: thmain}. Namely, in Lemma \ref{qwi: lemtransformcond}, we will now have the extra term $\lambda |  \int^{\infty}_{-\infty} W(x + y) \phi(\lambda y)  \dif y|$ in \eqref{qwi: resulttransformcond} which needs to be estimated.  We use the compact support of $W$ and $\phi(y) \ll \exp(-|y|^{\gamma'})$ for some fixed $ \gamma < \gamma' < 1$ to see that $\int^{\infty}_{-\infty} \lambda |W(x+y) \phi(\lambda y)| \dif y \ll \exp(-\lambda x^{\gamma})$ as $x \rightarrow \infty$ for any $\lambda \geq 1$. Therefore, by exploiting that the original Laplace transform $\mathcal{L}\{\tau; \beta + it \} - \mathcal{L}\{F;it\} \in \mathcal{T}_\ast$, Lemma \ref{qwi: lemtransformcond} now gives
$$ \lambda \bigg|  \int^{\infty}_{-\infty} S(x + y) \phi(\lambda y)  \dif y\bigg| \ll E_{\ast}(x,\lambda) + \exp(-\lambda x^{\gamma}), \ \ \ x \geq X.
$$ 
When one inserts this estimate now in Lemmas \ref{qwi: taublem} and \ref{qwilem: taublem}, the extra term $\exp(-\lambda x^{\gamma})$ becomes harmless as it can be absorbed in $F(x)/\lambda$ because $F(x) \gg \exp(-x^{\gamma})$ for $x \geq X$.

Therefore, we are justified to apply Theorem \ref{qwi: thmain} to $S$ and we find \eqref{qwi: thusefulmain} by rewriting $\tau$ in terms of $S$. This concludes the proof.
\end{proof}

We now provide the corresponding corollary for some concrete functions $F$ arising frequently in applications. Let $\rho \in \mathbb{C}$ and $\theta > 0$. We consider
\begin{equation} \label{qwi: eqdeffa} F_{\rho,\theta}(x) = x^{\rho} - G_{\rho,\theta}(x), \ \ \ G_{\rho,\theta}(x) = \begin{cases} x^{\rho}e^{-\theta x} \displaystyle{\sum_{j = 0}^{\lceil -\Re \rho \rceil -1}} \frac{\theta^{j}}{j!} x^j  & \text{if } \Re \rho < 0,\\
0 & \text{otherwise}.
\end{cases}
\end{equation}
The function $G_{\rho,\theta}$ is mainly an auxiliary function, decaying very fast at $\infty$ as to not interfere significantly with the asymptotic $x^{\rho}$ near $\infty$. This auxiliary function $G_{\rho,\theta}$ ensures that $F_{\rho,\theta}$ remains bounded near $x = 0$, while the Laplace transform remains explicitly computable. It is
\begin{align*} \mathcal{L}\{F_{\rho, \theta} ; s\} & = \frac{\Gamma(\rho + 1)}{s^{\rho  + 1}} - \mathcal{G}_{\rho, \theta}(s), \ \ \ \text{if } \rho \notin \{-1, -2, -3, \dots \}, \\
\mathcal{G}_{\rho, \theta}(s) & = \displaystyle{\sum_{j = 0}^{\lceil -\Re \rho \rceil -1}} \frac{\theta^{j}}{j!} \frac{\Gamma(\rho + 1 + j)}{ (s+ \theta)^{\rho + j + 1} }, \ \ \ \text{if } \Re \rho < 0 \text{ and } \rho \notin \{-1, -2, -3, \dots \},\\
\mathcal{L}\{F_{-N, \theta} ; s\} & = \frac{s^{N-1} (-1)^{N-1}}{(N-1)!} \log\Big(\frac{s+ \theta}{s}\Big) - \mathcal{G}_{-N, \theta}(s), \ \ \ \text{if } N \in \mathbb{N}, \\
\mathcal{G}_{-N, \theta}(s) & = \sum_{j = 1}^{N-1} \frac{\theta^j}{j!} \frac{(-1)^{N-j}(s+ \theta)^{N-j-1}}{(N-j-1)!} \sum_{\ell = N-j}^{N-1} \frac{1}{\ell}, \ \ \ \text{if } N \in \mathbb{N} \text{ and } N \geq 2.
\end{align*}
In the undefined cases we set $\mathcal{G}_{ \rho,\theta} \equiv 0$.

We now state our theorem for power-type singularities (and some logarithmic ones) on the Laplace transform. The locations of the singularities are collected in $\mathcal{S}$, while the precise power in each of these locations $s'$ are found within the set $\mathcal{A}_{s'}$.

\begin{theorem} \label{qwi: thspeccor} Let $\tau: [0,\infty) \rightarrow \mathbb{R}$ be a locally integrable eventually non-decreasing function and $\beta, \theta >0$. Let $\mathcal{S}$ be a finite set of complex numbers for which $\overline{s_l} \in \mathcal{S}$ if $s_l \in \mathcal{S}$ and $\max\{\Re s_l :  s_l \in \mathcal{S}\} = 0$. With each $s_l \in \mathcal{S}$, we associate a finite set $\mathcal{A}_{s_l}$ of complex numbers for which $\overline{\rho_{m,l}} \in \mathcal{A}_{s_l}$ if $\rho_{m,l} \in \mathcal{A}_{\overline{s_l}}$. We set $\rho := \max \Re \rho_{m,l'}$ where the maximum runs over all $\rho_{m,l'} \in \mathcal{A}_{s_{l'}}$ where $s_{l'} \in \mathcal{S}$ and $\Re s_{l'} = 0$. Let $c_{s_l,\rho_{m,l}}$ be complex numbers such that $c_{s_l, \rho_{m,l}} = \overline {c_{\overline{s_l}, \overline{\rho_{m,l}}}}$.

Then, if 
\begin{equation}
 \mathcal{L}\{\tau, s + \beta\} - \sum_{\substack{s_l \in \mathcal{S} \\ \rho_{m,l} \in \mathcal{A}_{s_l}}} c_{s_l, \rho_{m,l}} \mathcal{L}\{F_{\rho_{m,l}, \theta};s - s_l\} 
\end{equation}
converges for $\Re s > 0$ and admits a continuous extension to $\Re s = 0$ with $g(t) := \mathcal{L}\{\tau; \beta + it\} -  \sum c_{s_l,\rho_{m,l}}\mathcal{L}\{F_{\rho_{m,l},\theta}; it - s_l\}  \in \mathcal{T}_\ast$, then 
\begin{equation}  \label{qwi: eqspeccorsimp}
 \tau(x) \ll e^{\beta x} \Big(x^{\rho} + E_{\ast}(x,1)\Big), \ \ \ x \rightarrow \infty.
\end{equation}

Furthermore, if, additionally, $E_{\ast}(x,1) \ll x^{\rho}$, then as $x \rightarrow \infty$,
\begin{equation}
 \tau(x) = e^{\beta x} \sum_{\substack{s_l \in \mathcal{S} \\ \rho_{m,l} \in \mathcal{A}_{s_l}}} c_{s_l,\rho_{m,l}} e^{s_l x} x^{\rho_{m,l}} + O\bigg(e^{\beta x} \inf_{\lambda \geq 1} \left\{E_{\ast}(x,\lambda) + \frac{x^{\rho}}{\lambda} + e^{-\theta x}\right\}  \bigg).
\end{equation}
\end{theorem}
Again we note that the $O$-constants here are allowed to depend on the function $\tau$.
\begin{proof} Let $M$ be a sufficiently large real number. We apply Theorem \ref{qwi th:cor1} to $\tilde{\tau}(x) := \tau(x) + M e^{\beta x} F_{\rho,\theta}(x)$ and $\tilde{F}$ which is defined $\tilde{F}(x) := M F_{\rho,\theta}(x) + \sum_{s_l \in \mathcal{S}, \rho_{m,l} \in \mathcal{A}_{s_l}} c_{s_l, \rho_{m,l}} e^{s_l x}F_{\rho_{m,l},\theta}(x)$ for $x \geq \theta$ and $0$ otherwise. We observe that $\tilde{F}$ is real-valued because of the conditions on $c_{s_{l}, \rho_{m,l}}$. Furthermore, since $\rho$ is real, we can ensure by taking $M$ large enough that \eqref{eqqwi: hypFcor} holds for $\tilde{F}$ (for certain $C$ and $\alpha$). Since $\beta > 0$, it is clear that $\tilde{\tau}$ remains eventually non-decreasing.  

We can therefore apply Theorem \ref{qwi th:cor1}. The results now follow from subtracting $M e^{\beta x} F_{\rho,\theta}(x)$ from $\tilde{\tau}$ and noting that $G_{\rho_{m,l}, \theta}(x) \ll \exp(-\theta x)$ and $\Re s_l \leq 0$ for $s_l \in \mathcal{S}$.
\end{proof}

\begin{remark} Thus, Theorem \ref{qwi: thspeccor} provides a quantified Wiener-Ikehara theorem if the Laplace transform exhibits singularities of the form $s^{-\rho}$ and $s^{N} \log s$ (after a translation), where $\rho \in \mathbb{C}$ is not a non-negative integer and with $N$ a natural number or $0$. Note that, if $\rho$ were to be a non-negative integer, then $s^{-\rho}$ would not be singular at $s = 0$. 

The purpose of including the function $G_{\rho,\theta}$ in the definition of $F_{\rho, \theta}$ in \eqref{qwi: eqdeffa} in case $\Re \rho < 0$ is two-fold. First and foremost, it eliminates the singularity of $x^{\rho}$ at $x= 0$ without changing too much the behavior as $x \rightarrow \infty$. This ensures that the Laplace transform of $F_{\rho,\theta}$ can be explicitly computed, while the lack of a singularity at $x = 0$ is also helpful in the verification of \eqref{eqqwi: hypFcor}. Secondly, due to the inclusion of the function $G_{\rho, \theta}$, one has $\mathcal{L}\{F_{\rho, \theta};s\} \ll 1/s$ as $|s| \rightarrow \infty$, say in the half-plane $\Re s \geq - \theta/2$. This guarantees that potential good bounds\footnote{Consider for example the hypothesis $\mathcal{T}_{An}(G,H)$ where $G \equiv 1$ and $H(t) = \max\{1, t^{-1}\}$ in Theorem \ref{qwi: thspeccor} (and pick $\theta \geq 2$).} for the Laplace transform of $\tau$ in Theorem \ref{qwi: thspeccor} do not get compromised after subtracting the Laplace transform of $F_{\rho,\theta}$.
\end{remark}

In some applications, it is more natural to work with the Laplace-Stieltjes transform than the Laplace transform. Obviously the two transforms are connected through the formula $\int^{\infty}_{0} e^{-sx} \dif \tau(x) = s\int^{\infty}_0 \tau(x) e^{-sx} \dif x$ (on a suitable half-plane), but this extra factor $s$ may change the nature of the singularities (and thus of the main terms in the asymptotic expansion of $\tau$) if one transfers from the Laplace transform to the Laplace-Stieltjes transform. For convenience of the reader, we provide the corresponding theorem if one instead has a finite number of singularities of the form $s^{\rho}$ or $s^{N} \log s$ (after a translation) on the Laplace-Stieltjes transform.

\begin{theorem} \label{qwi: thcorlstiel} Let $\tau: [0,\infty) \rightarrow \mathbb{R}$ be a locally integrable eventually non-decreasing function and $\beta, \theta >0$. Let $\mathcal{S}$, $\mathcal{A}_{s_l}$, $\rho$ and $c_{s_l, \rho_{m,l}}$ be as in the statement of Theorem \ref{qwi: thspeccor}.

Then, if 
\begin{equation}
 \mathcal{L}\{\tau, s + \beta\} - \sum_{\substack{s_l \in \mathcal{S} \\ \rho_{m,l} \in \mathcal{A}_{s_l}}} c_{s_l, \rho_{m,l}} \frac{\mathcal{L}\{F_{\rho_{m,l}, \theta};s - s_l\} }{s + \beta}
\end{equation}
converges for $\Re s > 0$ and admits a continuous extension to $\Re s = 0$ with $g(t) := \mathcal{L}\{\tau; \beta + it\} -  \sum c_{s_l,\rho_{m,l}}\mathcal{L}\{F_{\rho_{m,l},\theta}; it - s_l\}/(\beta + it)  \in \mathcal{T}_\ast$, then \eqref{qwi: eqspeccorsimp} holds. Furthermore, if, additionally, $E_{\ast}(x,1) \ll x^{\rho}$, then as $x \rightarrow \infty$,
\begin{equation} 
 \tau(x) = \sum_{\substack{s_l \in \mathcal{S} \\ \rho_{m,l} \in \mathcal{A}_{s_l}}} c_{s_l,\rho_{m,l}} H_{s_l, \rho_{m,l}, \beta, \theta}(x) + O\bigg(e^{\beta x} \inf_{\lambda \geq 1} \left\{E_{\ast}(x,\lambda) + \frac{x^{\rho}}{\lambda}\right\}  \bigg),
\end{equation}
where
\begin{equation}
 H_{s_l, \rho, \beta, \theta}(x) = \int^{x}_{0} F_{\rho,\theta}(u) e^{(s_l + \beta) u } \dif u.
\end{equation}
\end{theorem}
\begin{proof} The proof is similar to that of Theorem \ref{qwi: thspeccor}, but by replacing, where appropriate, $e^{s_l x} F_{\rho_{m,l}, \theta}(x)$ with $e^{-\beta x} H_{s_l, \rho_{m,l}, \beta, \theta}(x)$ as $\mathcal{L}\{e^{-\beta x} H_{s_l, \rho, \beta, \theta}(x) ;s \} = \mathcal{L}\{F_{\rho, \theta} ; s- s_l\}/ (s + \beta)$. That property \eqref{eqqwi: hypFcor} indeed holds for the adjusted $\tilde{F}$ follows from $e^{s_l x} F_{\rho_{m,l},\theta}(x) \asymp e^{-\beta x} H_{s_l, \rho_{m,l}, \beta, \theta}(x)$. 
\end{proof}

As an illustration for how to apply these theorems in practice, we discuss how some zero-free regions of the Riemann zeta function, together with some bounds inside these regions, lead to quantified versions of the prime number theorem.

\begin{example} We shall discuss three functions whose asymptotics are often referred to as the prime number theorem. We shall assume throughout that $\zeta(s)$, defined initially as $\sum_{n = 1}^{\infty} n^{-s}$ on $\Re s > 1$, admits an analytic extension to $\Re s \geq 1 - c/(\log (2 + |\Im s |))^{\gamma}$, for some $c, \gamma > 0$, with the exception of a simple pole with residue $1$ at $s =1$. For the three functions we shall assume specific bounds tailored to the function of interest. These bounds will be natural in the sense that, at the cost of slightly reducing the region of validity, they can be derived from an appropriate zero-free region via the Borel-Carath\'eodory lemma, see for example \cite[Lemma 6.2]{MV07}, and the initial rough bound $|\zeta(s)| \ll |s|$, say as $|s| \rightarrow \infty$ in $\Re s > 0$. We however do not claim that this procedure is optimal; enlarging the region where adequate bounds for the respective Laplace transforms are valid should lead to better error terms. Our goal here is only to illustrate how to apply the Tauberian theorem by inserting some concrete bounds and regions.

The first function we consider is the usual prime counting function $\pi(x)$. We shall assume here that the bound $|\log \zeta(s) | \ll \log |s|$ holds as $|s| \rightarrow \infty$ in the region $\Re s \geq 1 - c/(\log (2 + |\Im s |))^{\gamma}$, for some $c,\gamma > 0$. We intend to apply Theorem \ref{qwi: thcorlstiel} to $\tau(x) = \pi(e^{x})$. After some standard calculations one finds that
$$ \mathcal{L}\{\pi(e^x);s\} = \frac{\log \zeta(s) + B(s)}{s}, \ \ \ \Re s > 1,
$$
where $B(s)$ is an analytic function that is bounded uniformly on the half-plane $\Re s \geq 1/2 + \varepsilon$, for each $\varepsilon > 0$. We thus wish to apply the Tauberian theorem with $\mathcal{T}_{\mathrm{An}}$ and $G(t) = c^{-1}\log^{\gamma}(t + 2)$ and $H(t) = C\log (t + 2) / t$, for some constant $C$. We only have to subtract the singularity at $s = 1$, which is of type $\log(1/(s-1))/s$, because $\zeta(s)$ has a simple pole at $s = 1$. Therefore, we select $\mathcal{S} = \{0\}$, $\mathcal{A}_{0} = \{-1\}$, $c_{0,-1} = 1$ and $\beta = 1$ and $\theta = 2$, say. Note also that the (the boundary function of) Laplace transform $\mathcal{L}\{\tau; s + 1\}$ remains in $\mathcal{T}_{\mathrm{An}}(G,H)$ after subtracting $\mathcal{L}\{F_{-1,2} ; s\}/(s+1)$ as it is $O(1/s)$ as $|s| \rightarrow \infty$ in the region $\Omega$, while also the term $B(s) / s \ll s^{-1}$ is harmless. Since $\pi(e^{x})$ is non-decreasing, we can apply Theorem \ref{qwi: thcorlstiel} to obtain
\begin{align*} \pi(e^x) & = \int^{x}_0 \frac{e^{u}- e^{-u}}{u} \dif u \\
& \ \ \ + O\bigg(e^{x} \inf_{\lambda \geq 2c/\log^{\gamma}(2)} \bigg\{\exp\bigg(-\frac{cx}{\log^{\gamma}(\lambda + 2) + c\kappa \lambda^{-1}}\bigg) \log^{2}(\lambda + 2) + \frac{1}{x \lambda} \bigg\}  \bigg),
\end{align*}
which after a change of variables and selecting the optimal choice $\log \lambda = (cx)^{1/(\gamma + 1)} - (\gamma + 3)(\gamma + 1)^{-2} \log x + O(1)$ becomes
$$ \pi(x) = \operatorname*{Li}(x) + O\bigg(x \exp\Big(- (c\log x)^{1/(\gamma + 1)} - \frac{\gamma^2 + \gamma - 2}{(\gamma + 1)^2} \log \log x\Big) \bigg), \ \ \ x \rightarrow \infty,
$$
where $\operatorname*{Li}(x) = \int^{x}_2 \log^{-1} u \ \mathrm{d}u$ stands for the logarithmic integral.

Since $H(t)$ is of subpolynomial growth, it is technically advantageous here to work with the error $E_{\mathrm{AnA}}$ from \eqref{qwi: eqmodifiedean} instead of $E_{\mathrm{An}}$. This essentially amounts to replacing the $\log^2 (\lambda + 2)$ factor in the infimum with $\log (\lambda + 2)$ since $n$ can be chosen as $n \asymp x / G(\lambda)$ which delivers $ n \gg \log \lambda$ in view that $\lambda$ will be chosen to obey $\log x + \log \lambda + \log \log \lambda \asymp cx/G(\lambda)$. This gives the following slightly better estimate:
 $$ \pi(x) = \operatorname*{Li}(x) + O\bigg(x \exp\Big(- (c\log x)^{1/(\gamma + 1)} - \frac{\gamma^2 + \gamma - 1}{(\gamma + 1)^2} \log \log x\Big) \bigg), \ \ \ x \rightarrow \infty.
$$

As the second function, we consider $M(x) := \sum_{n \leq x} \mu(n)$, where $\mu(n)$ stands for the M\"obius function, that is the multiplicative function defined by $\mu(p) = -1$ on primes and $\mu(p^{k}) = 0$ on the higher ($k \geq 2$) order prime powers. In this case we assume that $1/|\zeta(s)| \ll (1 +|\Im s|)^{A}$ in the region $\Re s \geq 1 - c/(\log (2 + |\Im s |))^{\gamma}$, for some $c, A, \gamma > 0$. For convenience, we also assume that $\zeta(s)$ satisfies the same bound in the same region, although that is strictly speaking not necessary\footnote{Instead of Theorem \ref{qwi: thcorlstiel}, which requires the explicit behavior $\mathcal{T}_{\mathrm{An}}$ of the Laplace transform, one may use the more precise \emph{exact} behavior of the Laplace transform, see \cite[Remark 4.1]{Deb-qik}. The desired exact behavior for $\zeta(s)/s$ would then immediately follow from the asymptotic $\lfloor e^{x}\rfloor = e^{x} + O(1)$.}. 

The Laplace transform of $M(e^x)$ is $1/(s\zeta(s))$. Yet, we cannot immediately apply Theorem \ref{qwi: thcorlstiel} here because $M(e^{x})$ is not non-decreasing. We therefore consider instead $\tau(x) = M(e^{x}) + \lfloor e^{x}\rfloor = \sum_{n \leq e^{x}} (\mu(n) + 1)$, which is non-decreasing and whose Laplace transform is $(\zeta(s) + \zeta(s)^{-1})/s$. After subtracting the simple pole at $s = 1$ with residue $1$, which is encoded by the Laplace transform of $F_{0,\theta}$, one finds that (the boundary function of) $\mathcal{L}\{\tau; s+ 1\} - \mathcal{L}\{F_{0,\theta};s\}/(s+1)$ belongs to the class $\mathcal{T}_{\mathrm{An}}$ with $G(t) = c^{-1}\log^{\gamma}(t + 2) $ and $H(t) = C t^{A-1}$, for some $C$. Applying
Theorem \ref{qwi: thcorlstiel} then yields,
$$ \tau(x) = e^{x} - 1 + O\bigg(e^{x} \inf_{\lambda \geq 2c/\log^{\gamma} 2} \bigg\{\exp\bigg(-\frac{cx}{\log^{\gamma}(\lambda + 2) + c\kappa \lambda^{-1}}\bigg) (\lambda + 2)^{A} + \frac{1}{\lambda} \bigg\}  \bigg),
$$
which after a change of variables, inserting the asymptotic $\lfloor e^{x}\rfloor = e^{x} + O(1)$ and choosing $\log \lambda = (cx/(A+1))^{1/(\gamma  +1)}$ becomes
$$ M(x) = O\bigg(x \exp\Big(- c^{1/(\gamma + 1)} (A+1)^{-1/(\gamma + 1)} \log^{1/(\gamma + 1)} x\Big) \bigg), \ \ \ x \rightarrow \infty.
$$  
 
The third and final function we consider is the Chebyshev function $\psi(x) = \sum_{p^{k} \leq x} \log p$. This weighted prime counting function is often treated instead of $\pi$ because its Dirichlets series admits a simple pole at $s = 1$, which is somewhat easier to handle than the logarithmic singularity that we had encountered for the first function. Here we assume that $|\zeta'(s)/ \zeta(s)| \ll |s|^{A}$ as $|s| \rightarrow \infty$ in the region $\Re s \geq 1 - c/(\log (2 + |\Im s |))^{\gamma}$, for some $c, A, \gamma > 0$. We consider the non-decreasing function $\tau(x) = \psi(e^x)$ whose Laplace transform is $-\zeta'(s)/ (s\zeta(s))$. It admits a simple pole at $s = 1$ with residue $1$. Thus, a similar calculation as for the second function delivers
$$ \psi(x) = x + O\bigg(x \exp\Big(- c^{1/(\gamma + 1)} (A+1)^{-1/(\gamma + 1)} \log^{1/(\gamma + 1)} x\Big) \bigg), \ \ \ x \rightarrow \infty,
$$
under the mentioned assumptions.
\end{example}

\begin{remark} The connection between the zero-free region and the error in the prime number theorem has been extensively studied in the literature, even for much more general zero-free regions, see for example \cite{broucke2025, pintz1980, revesz2025} for the function $\psi$. The method discussed in the references is somewhat different than the one discussed here and implicitly exploits some additional analytic information about the $\zeta$ function; one uses the analytic continuation of $\zeta$ in a larger half-plane and even some information about the distribution of zeros of $\zeta$ including some zero-density estimates. On the other hand that method does not (explicitly) use estimates for $\zeta(s)$ in the region of analytic continuation. Both approaches are essentially optimal with respect to the assumptions used on the $\zeta$ function. The method from \cite{broucke2025, pintz1980, revesz2025} delivers in many instances, the best possible error term for $\psi$, given this extra analytic information $\zeta$ is taken into account, while our Wiener-Ikehara approach should be optimal if one only has the information that $\zeta(s)$ has a meromorphic extension (with respective bounds on $\zeta'(s)/\zeta(s)$) in the region $\Re s \geq 1 - c/(\log (2 + |\Im s |))^{\gamma}$.   

We however do not claim that a Wiener-Ikehara approach cannot capture the analytic information that is exploited in \cite{broucke2025, pintz1980, revesz2025}. The starting point of the method from \cite{pintz1980} is the Riemann-von Mangoldt formula providing a formula for $\psi$ in terms of the zeros of the $\zeta$ function. Then one can use some results about the distribution of these zeros to estimate the sum originating from these zeros. One might anticipate that an effective quantified Wiener-Ikehara theorem can incorporate the contribution of these zeros into the flexible function $F$ introduced in this paper, while the results on the distribution of these zeros can then be used for estimating this function $F$.
\end{remark}


\begin{thebibliography}{99}

\bibitem{aramaki1988} J.~Aramaki,  \emph{On an extension of the Ikehara Tauberian theorem,} Pacific J. Math. \textbf{133} (1988),  13--30. 

\bibitem{bateman-diamond} P.~T.~Bateman, H.~G.~Diamond, \textit{Asymptotic distribution of Beurling's generalized prime numbers,} in: Studies in Number Theory, pp. 152--210, Math. Assoc. Amer., Prentice-Hall, Englewood Cliffs, N.J., 1969.

\bibitem{broucke2025} F. Broucke, \emph{On the connection between zero-free regions and the error term in the Prime Number Theorem}, Anal. Math. (2026), available at https://doi.org/10.1007/s10476-026-00176-y.

\bibitem{chen-vindas} B. Chen, J. Vindas, \emph{Generalizations of Koga's version of the Wiener-Ikehara theorem}, Ann. Fen. Math. \textbf{51} (2026), 597--611

\bibitem{Deb-qik} G.~Debruyne, \emph{A general quantified Ingham-Karamata Tauberian theorem}, Adv. Math. \textbf{493} (2026), 110920, 66 pp.

\bibitem{Debruyne-VindasComplexTauberians} G.~Debruyne, J.~Vindas, \emph{Complex Tauberian theorems for Laplace transforms with local pseudofunction boundary behavior,} J. Anal. Math. \textbf{138} (2019), 799--833.

\bibitem{Debruyne-VindasGeneralWI} G.~Debruyne, J.~Vindas, \emph{Generalization of the Wiener-Ikehara theorem,} Illinois J. Math. \textbf{60} (2016), 613--624.

\bibitem{d-v} G.~Debruyne, J.~Vindas, \emph{On general prime number theorems with remainder,} in: Generalized Functions and Fourier Analysis, pp. 79-94. Oper. Theory Adv. Appl., Vol. 260, Springer, 2017.

\bibitem{Debruyne-VindasPNTEquivalences} G. Debruyne, J. Vindas, \emph{On PNT equivalences for Beurling numbers,} Monatsh. Math. \textbf{184} (2017), 401--424.

\bibitem{d-v-l1} G.~Debruyne, J.~Vindas, \emph{On Diamond's $L^1$-criterion for asymptotic density of Beurling generalized integers,} Michigan Math. J. \textbf{68} (2019), 211--223.

\bibitem{delange1954}H.~Delange, \emph{G\'{e}n\'{e}ralisation du th\'{e}or\`{e}me de Ikehara,} Ann. Sci. Ecole Norm. Sup. \textbf{71} (1954), 213--242. 

\bibitem{hartlappthesis} M. Hartlapp, \emph{Quantified versions of the Ingham-Karamata theorem and the Wiener-Ikehara theorem}, PhD thesis, 2024, available at https://nbn-resolving.org/urn:nbn:de:bsz:14-qucosa2-960968.

\bibitem{katsuda-tsunada} A. Katsuda, T. Sunada, \emph{Closed orbits in homology classes,} Inst. Hautes Études Sci. Publ. Math. No. \textbf{71} (1990), 5--32.

\bibitem{koga} T. Koga, \emph{A Tauberian theorem of Wiener-Ikehara type and its applications,} J. Fourier Anal. Appl. \textbf{27} (2021), Paper No. 18, 21 pp.

\bibitem{korevaarbook} J.~Korevaar, \textit{Tauberian theory. A century of developments}, Grundlehren der Mathematischen Wissenschaften, 329, Springer-Verlag, Berlin, 2004.
 
\bibitem{malliavin} P.~Malliavin, \emph{Sur le reste de la loi asymptotique de r\'{e}partition des nombres premiers g\'{e}n\'{e}ralis\'{e}s de Beurling,} Acta Math. \textbf{106} (1961), 281--298. 

\bibitem{MV07} H.~L.~Montgomery, R.~C.~Vaughan,
\textit{Multiplicative number theory. I. Classical theory},
Cambridge Studies in Advanced Mathematics, Vol.~97, University
Press, Cambridge, 2007.

\bibitem{nyman} B.~Nyman, \emph{A general prime number theorem,} Acta Math. \textbf{81} (1949), 299--307.

\bibitem{petkov-stoyanov} V. Petkov, L. Stoyanov, \emph{Sharp large deviations for hyperbolic flows,} Ann. Henri Poincar\'e \textbf{21} (2020), 3791--3834.

\bibitem{p-t-z} L. Pierce, C. Turnage-Butterbaugh, A. Zaman, \emph{A guide to Tauberian theorems for arithmetic applications,} Ess. Number Th. \textbf{5} (2026), 271--374

\bibitem{pintz1980} J. Pintz, \emph{On the remainder term of the prime number formula II. On a theorem of Ingham,} Acta Arith. \textbf{37} (1980), 209--220.

\bibitem{revesz2025} Sz. Révész, \emph{The Method of Pintz for the Ingham Question About the Connection of Distribution of $\zeta$-Zeroes and Order of the Error in the PNT in the Beurling Context}, Michigan Math. J. \textbf{75} (2025), 761--806.

\bibitem{revesz-roton} Sz.~Gy.~R\'{e}v\'{e}sz, A. de Roton, \emph{Generalization of the effective Wiener-Ikehara theorem,} Int. J. Number Theory \textbf{9} (2013), 2091--2128. 

\bibitem{s-v} J.-C.~Schlage-Puchta, J.~Vindas, \textit{The prime number theorem for Beurling's generalized numbers. New cases}, Acta Arith. \textbf{153} (2012), 299--324.

\bibitem{shubin2001} M.~A.~Shubin, \emph{Pseudodifferential operators and spectral theory,} Second edition, Springer-Verlag, Berlin, 2001.

\bibitem{tenenbaumbook} G.~Tenenbaum, \emph{Introduction to analytic and probabilistic number theory,} American Mathematical Society, Providence, RI, 2015.

\bibitem{tranoy-vindas} Y. Tranoy, J. Vindas, \emph{A Wiener-Ikehara type theorem and its application to Chebyshev bounds for Beurling primes}, Anal. Math. (2026), available at https://doi.org/10.1007/s10476-026-00183-z.

\bibitem{vladimirov} V.~S.~Vladimirov, \textit{Methods of the theory of generalized functions,} Analytical Methods and Special Functions, 6, Taylor \& Francis, London, 2002.

\bibitem{zemanian} A.~H.~Zemanian, \emph{Distribution theory and transform analysis. An introduction to generalized functions, with applications}, Second edition, Dover Publications, Inc., New York, 1987. 


\bibitem{zhang2014} W.-B.~Zhang, \emph{Wiener-Ikehara theorems and the Beurling generalized primes,} Monatsh. Math. \textbf{174} (2014), 627--652.

\bibitem{zhang-exact} W.-B.~Zhang, \emph{Exact Wiener-Ikehara theorems,} Acta Arith. \textbf{187} (2019), 357--380. 

\end{thebibliography}
\end{document}